\documentclass[a4paper,14pt]{article}

\usepackage{mathtools, nicematrix}
\usepackage[left=20mm, top=20mm, right=20mm, bottom=20mm, nohead]{geometry}
\usepackage[T1,T2A]{fontenc}
\usepackage[utf8]{inputenc}
\usepackage[english]{babel}
\usepackage{amsmath}
\usepackage{amsfonts,amssymb}
\usepackage{amsthm}
\usepackage[electronic]{ifsym}
\usepackage[all]{xy}
\usepackage{authblk}
\usepackage{graphicx}%
\usepackage{multirow}%
\usepackage{amsmath,amssymb,amsfonts}%
\usepackage{amsthm}%
\usepackage{mathrsfs}%
\usepackage[title]{appendix}%
\usepackage{xcolor}%
\usepackage{textcomp}%
\usepackage{manyfoot}%
\usepackage{booktabs}%
\usepackage{algorithm}%
\usepackage{algorithmicx}%
\usepackage{algpseudocode}%
\usepackage{listings}%

\usepackage{hyperref}
\usepackage[nameinlink, capitalize]{cleveref}

\theoremstyle{definition}
\newtheorem{definition}{Definition}[section]

\theoremstyle{theorem}
\newtheorem{theorem}{Theorem}[section]

\theoremstyle{lemma}
\newtheorem{lemma}{Lemma}[section]

\theoremstyle{proposition}
\newtheorem{proposition}{Proposition}[section]

\theoremstyle{remark}

\theoremstyle{corollary}

\theoremstyle{problem}

\theoremstyle{hypothesis}

\newcommand{\N}{\mathbb{N}}

\begin{document}
\title{Axiomatization of B{\"u}chi arithmetic}

\author{Konstantin Kovalyov\footnote{This work was supported by Theoretical Physics and Mathematics Advancement Foundation “BASIS” as a part of project №22-7-2-32-1.}}

\maketitle

\begin{abstract}
    In this paper we introduce an axiomatization of B{\"u}chi arithmetic, i.e., of the elementary theory of natural numbers in the language with addition and function $V_p(a) = p^k$ such that $p^k | a$ and $p^{k + 1} \nmid a$.
\end{abstract}

\section{Introduction}

We consider B{\"u}chi arithmetic $\mathsf{BA}_p = Th(\mathbb N, S, +, 0, V_p)$, where $S(n) = n + 1$ and $V_p(a) = p^k$ such that $p^k | a$ and $p^{k + 1} \nmid a$ ($p \geqslant 2$ is some integer). The structure $(\mathbb N, S, +, 0, V_p)$ is automatic and, in fact, universal for automatic structures under first-order interpretability (see \cite[Theorem 4.5]{Blumensath2004}). Therefore, $\mathsf{BA}_p$ is decidable, but it seems that no explicit axiomatization of this theory is known. Our aim is to construct such a one. Although our axiomatization is not quite natural, it can be a first step towards a better axiomatization of B{\"u}chi arithmetic. Our axiomatization consists of a few simple axioms for $S$, $+$ and $V_p$ and the following scheme for all formulas $\varphi(x)$ with exactly one free variable:
$$\exists x \varphi(x) \to \exists x \leqslant \underline{n_{\varphi}}\: \varphi(x),$$
where $n_\varphi$ is chosen such that the formula above holds in the standard model. Here $\underline{n}$ denotes the numeral $S^n(0)$. Our, rather rough,  estimate on $n_\varphi$ is $p^{2_{|\varphi|}^{3}}$, where $2_m^k$ is iterated exponent ($2_0^k = k, 2_{m + 1}^k = 2^{2_m^k}$) and $|\varphi|$ stands for the length of $\varphi$.

One of the main features of B{\"u}chi arithmetic is that definable sets are exactly those base-$p$ notations of which are recognizable by finite automata. 

\begin{theorem}[\cite{Buchi}, \cite{bruyare}]
    A set $A \subseteq \N^k$ is definable in $(\N, S, +, 0, V_p)$ iff the set of $p$-ary expansions of elements of $A$ is recognizable by a finite automaton.
\end{theorem}

We will use this result to establish the estimate above for $n_\varphi$. Also,  this theorem implies that $\mathsf{BA}_p$ is decidable.

In some sense, B{\"u}chi arithmetic is more complex than Presburger arithmetic, for example, not every formula is equivalent to an $\exists$-formula in $\mathsf{BA}_p$ (see \cite{haase}). So, we are not able to find an   axiomatization of $\mathsf{BA}_p$ similar to that of Presburger arithmetic. Nevertheless, the following result holds.

\begin{theorem}[\cite{VILLEMAIRE}, \cite{haase}]
    Every formula $\varphi$ is equivalent to an $\exists \forall$-formula in $\mathsf{BA}_p$.
\end{theorem}

The problem of axiomatizing B{\"u}chi arithmetic has been considered earlier by Alexander Zapryagaev and there is the following negative result.

\begin{theorem}[\cite{zapryagaev}]
    A theory consisting of the following axioms is not complete:
        \begin{itemize}
            \item [(i)] axioms of Presburger arithmetic;
            \item [(ii)] $V_2(x) = 0 \leftrightarrow x = 0$;
            \item [(iii)] $\neg \exists y (x = y + y) \rightarrow V_2(x) = 1$;
            \item [(iv)] $x = t + t \rightarrow V_2(x) = V_2(t) + V_2(t)$.
        \end{itemize}
\end{theorem}

In particular, this theory does not prove the following true (in the standard model) formula:
$$V_2(x) = x \rightarrow \neg \exists y (x < y < x + x \wedge V_p(y) = y)$$
(which says <<there is no power of 2 between $2^k$ and $2^{k + 1}$>>).

\section{Notations and conventions}

For integer an $p \geqslant 2$, we denote by $V_p$ the function $\mathbb N \to \mathbb N$ such that $V_p(0) = 0$ and $V_p(n) = p^k$ for all $n > 0$, where $p^k | n$ and $p^{k + 1} \nmid n$. By $S$ we denote the function $n \mapsto n  + 1$.

By the \emph{standard model} we mean the structure $(\mathbb N, S, +, 0, V_p)$. To simplify notation it is usually denoted as $\mathbb N$. Arbitrary structures are denoted by $\mathcal M, \mathcal N, \dots$ and their domains by  $M, N, \dots$. 

\emph{Base-$p$ B{\"u}chi arithmetic} is the elementary theory of the standard model $(\mathbb N, S, +, 0, V_p)$. We denote this theory $\mathsf{BA}_p$. In the language of this theory, we denote by $\underline n$ the numeral $S(S(\dots S(0) \dots)) = S^n(0)$, where $n \in \N$. For $n \in \N, n > 0$ and $t$ a term, we denote by $n t$ the term $\underbrace{(\dots((t + t) + t) \dots)}_{n\text{ times}}$. 

By a \emph{finite automaton} (FA for short) over a finite alphabet $\Sigma$ (usually, $\Sigma$ will be $\{0, 1, \dots, p - 1\}$ or $\{0, 1, \dots, p - 1\}^k$) we mean a tuple $(\Sigma, Q, q_0, F, \Delta)$, where $Q$ is a finite set of states, $q_0 \in Q$ is an initial state, $F \subseteq Q$ is the set of final (or accepting) states, $\Delta \subseteq Q \times \Sigma \times Q$ is the set of transitions. By a \emph{determenistic finite automaton} (DFA for short) over a finite alphabet $\Sigma$ we mean a FA, where $\Delta$ is a function $Q \times \Sigma \to Q$ (i.e. for all $(q, a) \in Q \times \Sigma$ there exists a unique $q' \in Q$ such that $(q, a, q') \in \Delta$). For a FA $M$ we denote by $L(M)$ the language recognized by $M$. 

Denote by $\Sigma_p$ the alphabet $\{0, \dots, p - 1\}$. For $n \in \N$ we denote by $(n)_p \in \Sigma_p^*$ the $p$-ary expansion of $n$. That is, if $n = 0$, then $(n)_p = \Lambda$ (the empty word) and if $n = w_k p^k + \dots + w_0$, where $w_k, \dots, w_0 \in \{0, \dots, p - 1\}$ and $w_k \ne 0$, then $(n)_p = w_0 \dots w_k$. For an arbitrary $w \in \Sigma_p$ we denote by $[w]_p$ the number $w_{k - 1} p^{k - 1} + \dots + w_0$, where $w = w_0 \dots w_{k - 1}$ (we set $[\Lambda]_p = 0$). Moreover, for $w \in (\Sigma_p^k)^*$ we set $[w]_p = ([w^1]_p, \dots, [w^k]_p) \in \N^k$, where $(w^j)_i = (w_i)_j$. For example, for $p = 5$ and $k = 3$, we have $[(1, 2, 3)(2, 3, 1)(3, 1, 2)]_5 = (1 + 2 \cdot 5 + 3 \cdot 5^2, 2 + 3 \cdot 5 + 1 \cdot 5^2, 3 + 1 \cdot 5 + 1 \cdot 5^2) = (86, 42, 33)$. Of course, $[\cdot]_p$ is not one-to-one.

We say that $A \subseteq \N$ is {$p$-recognizable} if the language $L(A) := \{w \in \Sigma_p^* | [w]_p \in A\}$ is recognizable by a (determenistic) finite automaton. Moreover, we can generalize this notion to the subsets of $\N^k$. We say that $A \subseteq \N^k$ is {$p$-recognizable} if the language $L(A) := \{w \in (\Sigma_p^k)^* | [w]_p \in A\}$ is recognizable by a (determenistic) finite automaton.

We will use the following well-known facts about finite automata.

\begin{proposition}\label{prop_about_automata}
    \begin{enumerate}
        \item[$(i)$] If $M = (\Sigma, Q, q_0, F, \Delta)$ is DFA, then $M' := (\Sigma, Q, q_0, Q \setminus F, \Delta)$ is also DFA and, moreover, $L(M') = \Sigma^* \setminus L(M)$.
        
        \item[$(ii)$] If $M_0 = (\Sigma, Q_0, q_0, F_0, \Delta_0)$ and $M_1 = (\Sigma, Q_1, q_1, F_1, \Delta_1)$ are DFA, then $M := (\Sigma, Q_0 \times Q_1, (q_0, q_1), F_0 \times F_1, \Delta)$ is also DFA, where $\Delta((q, q'), a) = (\Delta_0(q, a), \Delta_1(q', a))$, and, moreover, $L(M) = L(M_0) \cap L(M_1)$.
        
        \item[$(iii)$] If $M = (\Sigma, Q, q_0, F, \Delta)$ is FA, then there if a DFA $M'$ with the set of states $\mathcal P (Q)$ such that $L(M') = L(M)$.
    \end{enumerate}
\end{proposition}

\section{Axiomatization}

We will denote our axiomatization by $\mathsf{T_{BA}}_p$.

\begin{definition}
    $\mathsf{T_{BA}}_p$ consists of the following axioms:
    \begin{itemize}
        \item [$(S0)$] $S x = S y \to x = y$;
        \item [$(S1)$] $0 \ne S x$;
        \item [$(S2)$] $x = 0 \vee \exists y (x = S y)$;
        \item [$(A0)$] $x + 0 = x$;
        \item [$(A1)$] $x + S y = S(x + y)$;
        \item [$(V0)$] $V_p(0) = 0$;
        \item [$(V1)$] $V_p(\underline{1}) = \underline{1}$;
        \item [$(V2)$] $V_p(p x) = p V_p(x)$;
        \item [$(V3)$] $\bigwedge\limits_{i = 1}^{p - 1} (V_p(p x + \underline{i}) = \underline{1})$;
        \item [$(Bound_\varphi)$] $\exists x \varphi(x) \to \exists x \leqslant \underline{n_{\varphi}}\: \varphi(x),$ for each formula $\varphi(x)$ with exactly one free variable,        
        where $n_\varphi = p^{2_{|\varphi|}^{3}}$.
    \end{itemize}
\end{definition}

First we argue that this axiomatization is sound.

\begin{theorem}\label{theorem_soundness}
    $(\mathbb N, S, +, 0, V_p) \vDash \mathsf{T_{BA}}_p$ (i.e. $\mathsf{BA}_p \vDash \mathsf{T_{BA}}_p$).
\end{theorem}

\begin{proof}
    It is clear that axioms $(S0)-(S2), (A0)-(A1), (V0)-(V3)$ holds in the standard model. It is enough to show that the scheme $(Bound_\varphi)$ holds in $\mathbb N$.

    The main idea is the following. Consider any formula $\varphi$ and an automaton $M_\varphi$ that recognizes the set, defined by $\varphi$. If there is a word, acceptable by $M_\varphi$ (i.e. $\exists x \varphi(x)$), then there is a path from an initial state to a final state, so, there is a word of length less or equal the number of states of $M_\varphi$, say $m$, acceptable by $M_\varphi$. Hence, there is $x\in \mathbb N$ with $p$-ary expansion of length at most $m$ (i.e. $\leqslant p^m$) such that $\varphi(x)$ holds. So, our task is to estimate the number of states of $M_\varphi$.

    Denote by $A_S$ the graph of the fucntion $S$, i.e. the set $\{(x, y) \in \N | S(x) = y\}$, similarly we define $A_{V_p}, A_+$ and $A_=$. It is not very hard to construct DFAs $M_S, M_{V_p}, M_+$ and $M_=$ recognizing $L(A_S), L(A_{V_p}), L(A_+)$ and $L(A_=)$ respectively, with no more $3$ states. 

    Firstly, let $\varphi(x_1, \dots, x_k)$ be a formula such that all occurrences of atomic formulas in it are of the form $x = y$, $S(x) = y$, $x + y = z$ and $V_p(x) = y$, where $x, y, z$ are distinct variables. We prove by induction on $\varphi$ that there is a DFA recognizing $L(\{(n_1, \dots, n_k) \in \N^k | \N \vDash \varphi(n_1, \dots, n_k)\})$ with no more than $2_{N_\varphi}^3$ states, where $N_\varphi$ is the number of boolean connectives $\wedge, \vee, \to$ and quantifiers in $\varphi$. We denote the latter set by $L_\varphi$.

    The base cases are already treated. 

    If $\varphi(x_1, \dots, x_k) = \neg \psi(x_1, \dots, x_k)$, then $L_\psi$ is recognized by a DFA $M = (\Sigma_p^k, Q, q_0, F, \Delta)$, $|Q| \leqslant 2_{N_\psi}^3$. By Proposition \autoref{prop_about_automata}, $L_\varphi$ is recognized by $(\Sigma, Q, q_0, Q \setminus F, \Delta)$ and $|Q| \leqslant  2_{N_\varphi}^3 = 2_{N_\psi}^3$.

    If 
    $$\varphi(x_1, \dots, x_k, y_1, \dots, y_l, z_1, \dots, z_m) = (\psi_0(x_1, \dots, x_k, y_1, \dots, y_l) \wedge \psi_1(y_1, \dots, y_l, z_1, \dots, z_m)),$$ 
    where $\{x_1, \dots, x_k, y_1, \dots, y_l\}$ are all distinct free variables in $\psi_0$, $\{y_1, \dots, y_l, z_1, \dots, z_m\}$ are all distinct free variables in $\psi_1$, then $L_{\psi_0}$ is recognized by a DFA 
    $$M_0 = (\Sigma_p^{k + l}, Q_0, q_0, F_0, \Delta_0),$$
    where $|Q_0| \leqslant 2_{N_{\psi_0}}^3$, and 
    $L_{\psi_1}$ is recognized by a DFA 
    $$M_1 = (\Sigma_p^{l + m}, Q_1, q_1, F_1, \Delta_1),$$ 
    where $|Q_1| \leqslant 2_{N_{\psi_1}}^3$. 
    Let 
    $M_0' := ((\Sigma_p)^{k + l + m}, Q_0, q_0, F_0, \Delta_0'),$ 
    where 
    $$\Delta_0'(q, (a_1, \dots, a_{k + l + m})) = \Delta_0(q, (a_1, \dots, a_{k + l}))$$ 
    and 
    $M_1' := ((\Sigma_p)^{k + l + m}, Q_1, q_1, F_1, \Delta_1'),$
    where 
    $$\Delta_1'(q, (a_1, \dots, a_{k + l + m})) = \Delta_1(q, (a_{k + 1}, \dots, a_{k + l + m})).$$ 
    Clearly, $L_\varphi = L(M_0') \cap L(M_1')$. Then, using Proposition \autoref{prop_about_automata}, we obtain that there is a DFA $M$ recognizing $L_\varphi$ with $|Q_0|\cdot |Q_1|$ states. But $|Q_0|\cdot |Q_1| \leqslant 2_{N_{\psi_0}}^3 \cdot 2_{N_{\psi_1}}^3 \leqslant 2_{N_{\psi_0} + N_{\psi_1} + 1}^3 = 2_{N_{\varphi}}^3$.

    The cases of $\varphi = (\psi_0 \vee \psi_1)$ and $\varphi = (\psi_0 \to \psi_1)$ can be treated similarly to the previous cases using $\varphi \leftrightarrow \neg(\neg \psi_0 \wedge \neg \psi_1)$ and $\varphi \leftrightarrow \neg(\psi_0 \wedge \neg \psi_1)$ respectively. 

    If $\varphi(x_1, \dots, x_k) = \exists x \psi(x_1, \dots, x_k, x)$, then $L_\psi$ is recognized by DFA $M = ((\Sigma_p)^{k + 1}, Q, q_0, F, \Delta)$, $|Q| \leqslant 2_{N_\psi}^3$. Let $M$ be a FA $((\Sigma_p)^{k}, Q, q_0, F, \Delta')$, where $\Delta' = \{(q, (a_1, \dots, a_{k}), q') \in Q \times (\Sigma_p)^{k} \times Q | \exists a \in \Sigma_p: (q, (a_1, \dots, a_{k}, a), q') \in \Delta\}$. Informally saying, we cleared the last symbol on each edge in our automaton. Clearly, $L_\varphi$ is recognized by $M'$, but $M'$ can be non-deterministic. To construct a DFA $M''$ recognizing $L_\varphi$ one can apply Proposition \autoref{prop_about_automata}, then the number of states in $M''$ is $2^{|Q|} \leqslant 2^{2_{N_\psi}^3} = 2_{N_\varphi}^3$.

    The case of $\varphi(x_1, \dots, x_k) = \forall x \psi(x_1, \dots, x_k, x)$ can be treated similarly using the equivalence $\varphi(x_1, \dots, x_k) \leftrightarrow \neg \exists x \neg \psi(x_1, \dots, x_k, x)$.

    Then, let $\varphi(x)$ be an arbitrary formula with exactly one free variable. Clearly, there is an equivalent formula $\tilde \varphi(x)$ such that all occurrences of atomic formulas in it are of the form $x = y$, $S(x) = y$, $x + y = z$ and $V_p(x) = y$, where $x, y, z$ are distinct variables. It is not very hard to see that $N_{\tilde \varphi} \leqslant |\varphi|$. Hence, $L(\{n \in \N| \N \vDash \varphi(n)\})$ is recognizable by a DFA with $\leqslant 2_{|\varphi|}^3$ states. 
\end{proof}

Further we need to prove completeness, i.e. that $\mathsf{T_{BA}}_p$ proves any sentence, which is true in the standard model. 

\begin{theorem}\label{theorem_completeness}
    $\mathsf{T_{BA}}_p \vDash \mathsf{BA}_p$, i.e. $\mathsf{T_{BA}}_p$ is complete.
\end{theorem}

We will gives two proofs of the result above: a syntactical and a model-theoretic ones.

\section{A model-theoretic proof of completeness}

We begin with the following general (and rather trivial) lemma.

\begin{lemma}\label{lemma_about_types}
    Let $T$ be a consistent theory in at most countable language, $p(\vec x)$ be a non-isolated (possibly incomplete) type in $T$ and $\psi$ be a consistent with $T$ sentence. Then $p(\vec x)$ is non-isolated in $T + \psi$.
\end{lemma}

\begin{proof}
    Let $T$, $p(\vec x)$ and $\psi$ be as in the statement. Suppose, aiming at a contradiction, there is a formula $\varphi(\vec x)$, which isolates $p(\vec x)$ in $T + \psi$, i.e.  $(T + \psi) + \exists \vec x \: \varphi(\vec x)$ is consistent and $T + \psi \vDash \forall \vec x (\varphi(\vec x) \to \eta(\vec x))$ for all $\eta(\vec x) \in p(\vec x)$. Then, for all $\eta(\vec x) \in p(\vec x)$, we have  $$T \vDash \psi \to \forall \vec x (\varphi(\vec x) \to \eta(\vec x)),$$
    $$T \vDash \forall \vec x (\psi \to (\varphi(\vec x) \to \eta(\vec x))),$$
    $$T \vDash \forall \vec x (\psi \wedge \varphi(\vec x) \to \eta(\vec x)).$$

    Since $p(\vec x)$ is non-isolated in $T$, it follows that $T + \exists \vec x (\psi \wedge \varphi(\vec x))$ is inconsistent. Then $(T + \psi) + \exists \vec x \: \varphi(\vec x)$ is inconsistent, a contradiction. 
\end{proof}

\begin{proof}[Proof of \autoref{theorem_completeness}.] Suppose that there is a sentence $\psi$, which holds in the standard model and $\mathsf{T_{BA}}_p \nvDash \psi$. Then $\mathsf{T_{BA}}_p + \neg \psi$ is consistent. Consider the type $p(x) = \{x \ne \underline{n} | n \in \mathbb N\}$. 

We argue that $p(x)$ is non-isolated in $\mathsf{T_{BA}}_p$. Suppose there is a formula $\varphi(x)$ such that $\mathsf{T_{BA}}_p + \exists x \: \varphi(x)$ is consistent and 
\begin{align*}
    \mathsf{T_{BA}}_p \vDash \forall x (\varphi(x) \to x \ne \underline{n}) \text{ for all }n \in N. && (*)
\end{align*} Consider an arbitrary model $\mathcal M \vDash \mathsf{T_{BA}}_p + \exists x \: \varphi(x)$. By $(Bound_\varphi)$ we have $\mathcal M \vDash \exists x \leqslant \underline{n_\varphi} \: \varphi(x)$. Since $\mathsf{T_{BA}}_p \vdash x \leqslant \underline{n_\varphi} \leftrightarrow \bigvee\limits_{k = 0}^{{n_\varphi}} (x = \underline{k})$, there is some natural $k \leqslant \underline{n_\varphi}$ such that $\mathcal M \vDash \varphi(\underline{k})$. But then, substituting $k$ for $n$ and $\underline{k}$ for $x$ in $(*)$, we obtain $\mathcal M \vDash \underline{k} \ne \underline{k}$, a contradiction.

So, $\mathsf{T_{BA}}_p + \neg \psi$ is consistent, $p(x)$ is non-isolated in $\mathsf{T_{BA}}_p$, hence, by Lemma \autoref{lemma_about_types}, $p(x)$ is non-isolated in $\mathsf{T_{BA}}_p + \neg \psi$. By the Omitting Type Theorem, there is a model $\mathcal M \vDash \mathsf{T_{BA}}_p + \neg \psi$ which omits the type $p(x)$. Since $\mathcal M$ omits $p(x)$, every $m \in M$ is of the form $\underline n$, $n \in \mathbb N$. Using the axioms $(S0)-(S2), (A0)-(A1), (V0)-(V3)$ it is easy to see that $+$ and $V_p$ are defined in a standard way on $M$, that is, $\mathcal M \cong (\mathbb N, S, +, 0, V_p)$. But then $\mathcal M \vDash \psi$, a contradiction.
\end{proof}

\section{A syntactic proof of completeness}

We need the following (again rather trivial) lemmata.

\begin{lemma}\label{lemma_about_terms}
    \begin{enumerate}
        \item [$(1)$] $\mathsf{T_{BA}}_p \vdash \underline{n} + \underline{m} = \underline{n + m}$ for all $n, m \in \mathbb N$;
        \item [$(2)$] $\mathsf{T_{BA}}_p \vdash V_p(\underline{n}) = \underline{V_p(n)}$ for all $n \in \mathbb N$;
        \item [$(3)$] $\mathsf{T_{BA}}_p \vdash \underline{n} \ne \underline{m}$ for all $n, m \in \mathbb N$, $n \ne m$;
        \item [$(4)$] if $t$ is a closed term and $n$ its value in the standard model, then $\mathsf{T_{BA}}_p \vdash t = \underline{n}$;
        \item [$(5)$] $\mathsf{T_{BA}}_p \vdash x \leqslant \underline{n} \leftrightarrow \bigvee\limits_{k = 0}^{{n}} (x = \underline{k})$ for all $n\in \mathbb N$.
    \end{enumerate}
\end{lemma}

\begin{proof}
    \begin{enumerate}
        \item [$(1)$] The proof can be found in \cite[Theorem 1.6]{hajek_pudlak_2017}, however, it is quite simple. We argue by induction on $m$. If $m = 0$, it is simply the axiom $(A0)$. For the induction step we have 
        \begin{align*}
            \mathsf{T_{BA}}_p &\vdash \underline{n} + \underline{m + 1} = \underline{n} + S(\underline{m});\\
            &\vdash \underline{n} + \underline{m + 1} = S(\underline{n} + \underline{m}); &(A1)\\
            &\vdash \underline{n} + \underline{m + 1} = S(\underline{n + m}); &(\text{induction hypothesis})\\
            &\vdash \underline{n} + \underline{m + 1} = \underline{n + m + 1}.
        \end{align*}
        
        \item [$(2)$] We argue by induction on $n$. If $n = 0$ (or $n = 1$), it is simply the axiom $(V0)$ (resp., $(V1)$). Suppose $n > 1$ and $\mathsf{T_{BA}}_p \vDash V_p(\underline{n'}) = \underline{V_p(n')}$ for $n' < n$. If $p | n$, then $n = p n'$, $n' < n$ and $V_p(n) = p V_p(n')$. Then we have
        \begin{align*}
            \mathsf{T_{BA}}_p &\vdash V_p(\underline{n}) = V_p(\underline{p n'});\\
            &\vdash V_p(\underline{n}) = V_p(p\underline{n'}); &((1) \text{ applied } p \text{ times})\\
            &\vdash V_p(\underline{n}) = p V_p(\underline{n'}); &(V2)\\
            &\vdash V_p(\underline{n}) = p \underline{V_p (n')}; &(\text{induction hypothesis})\\
            &\vdash V_p(\underline{n}) = \underline{p V_p (n')}; &((1) \text{ applied } p \text{ times})\\
            &\vdash V_p(\underline{n}) = \underline{V_p (n)}.
        \end{align*}

        If $p \nmid n$, then $n = p n' + i$, where $0 < i < p$. Then we have 
        \begin{align*}
            \mathsf{T_{BA}}_p &\vdash V_p(\underline{n}) = V_p(\underline{p n' + i});\\
            &\vdash V_p(\underline{n}) = V_p(p\underline{n'} + \underline{i}); &((1) \text{ applied } p + 1 \text{ times})\\
            &\vdash V_p(\underline{n}) = \underline{1}; &(V3).
        \end{align*}
        
        \item [$(3)$] Let $n, m \in \mathbb N$, $n \ne m$. W.l.o.g. $n < m$. Then, applying $(S0)$ $n$ times, we have $\mathsf{T_{BA}}_p \vdash \underline{n} = \underline{m} \to 0 = \underline{m - n}$. Applying $(S1)$, we obtain $\mathsf{T_{BA}}_p \vdash \underline{n} \ne \underline{m}$.
        
        \item [$(4)$] Trivial induction on $t$.

        \item [$(5)$] See \cite[Theorem 1.6]{hajek_pudlak_2017}.
    \end{enumerate}
\end{proof}

\begin{lemma}\label{lemma_about_quantifier_free_sentences}
    Let $\theta$ be a quantifier free sentence. If $\theta$ holds in the standard model, then $\mathsf{T_{BA}}_p \vdash \theta$, otherwise $\mathsf{T_{BA}}_p \vdash \neg \theta$.
\end{lemma}

\begin{proof}
    We argue by induction on $\theta$.

    Let $\theta$ be an atomic formula of the form $t = s$, where $t$ and $s$ are closed terms. Let $n, m \in \mathbb N$ be the values of $t, s$ in the standard model respectively. By Lemma \autoref{lemma_about_terms} (4), $\mathsf{T_{BA}}_p \vdash t = \underline{n}$ and $\mathsf{T_{BA}}_p \vdash s = \underline{m}$. If $\mathbb N \vDash t = s$, then $n = m$ and $\mathsf{T_{BA}}_p \vdash t = s$. If $\mathbb N \vdash t \ne s$, then $n \ne m$ and by Lemma \autoref{lemma_about_terms} (3), $\mathsf{T_{BA}}_p \vdash \underline{n} \ne \underline{m}$, hence, $\mathsf{T_{BA}}_p \vdash t \ne s$.  

    The cases of boolean connectives treat trivially.
\end{proof}

\begin{lemma}\label{lemma_about_quantifier_elimination_in_sentences}
    Let $\varphi$ be a sentence, then there is a quantifier free sentence $\theta$ such that $\mathsf{T_{BA}}_p \vdash \varphi \leftrightarrow \theta$.
\end{lemma}

\begin{proof}
    We argue by induction on $\varphi$. Replacing all occurrences of $\forall x$ with $\neg \exists x \neg$, we may assume that there is no occurrences of $\forall$ in $\varphi$.
    
    If $\varphi$ is atomic, we are done. The $\varphi$ cases of boolean connectives are trivial. 

    Suppose $\varphi = \exists x \psi(x)$. Then, by $(Bound_\psi)$ and Lemma \autoref{lemma_about_terms} (5), we obtain $\mathsf{T_{BA}}_p \vdash \varphi \leftrightarrow \bigvee\limits_{k = 0}^{n_\psi} \psi(\underline{k})$. By induction hypothesis, there are quantifier free sentences $\theta_0, \dots, \theta_{n_\psi}$ such that $\mathsf{T_{BA}}_p \vdash \psi(\underline{k}) \leftrightarrow \theta_k$ for all $k \leqslant n_\psi$. Then $\mathsf{T_{BA}}_p \vdash \varphi \leftrightarrow \bigvee\limits_{k = 0}^{n_\psi} \theta_k$.
\end{proof}

\begin{proof}[Proof of \autoref{theorem_completeness}.] 
    Suppose $\mathbb N \vDash \varphi$, where $\varphi$ is a sentence. Combining Lemma \autoref{lemma_about_quantifier_elimination_in_sentences} and Lemma \autoref{lemma_about_quantifier_free_sentences}, we obtain that $\mathsf{T_{BA}}_p \vdash \varphi$.
\end{proof}

\bibliographystyle{unsrt}
\bibliography{references}

\end{document}